\documentclass[a4paper,english, 12pt]{amsart}
\usepackage[margin=3cm]{geometry} % Added by Takehiko. To determine the size of margins

\usepackage[latin1]{inputenc}
\usepackage[bottom]{footmisc}
\usepackage{mathrsfs}
\usepackage{amssymb}
\usepackage{amsmath} 
\usepackage{amsthm}
\usepackage{commath}
\usepackage{mathtools}
\usepackage{amsfonts}
\usepackage{tikz}
\usepackage{bm}
\usepackage{accents}
\usepackage{tikz-cd}
\usetikzlibrary{matrix}
\usepackage{stmaryrd}
\usepackage{hyperref}
\usepackage[shortlabels]{enumitem}

\newcommand{\QQ}{\mathbb{Q}}

\newcommand{\RR}{\mathbb{R}}

\newcommand{\oF}{\overline F}
\newcommand{\ZZ}{\mathbb{Z}}

\newcommand{\doots}{,\dots ,}

\newcommand{\Akk}{\mathfrak A_4}

\newcommand{\OO}{\mathcal O}

\newcommand{\Fv}{F_v}

\newcommand{\sG}{\prescript{}{\sigma}G}

\newcommand{\sA}{\prescript{}{\sigma}A}

\newcommand{\Ov}{\mathcal O _v}

\newcommand{\vMF}{v\in M_F}

\newcommand{\muu}{\boldsymbol{\mu}}

\DeclareMathOperator{\Aut}{Aut}
\DeclareMathOperator{\Autt}{\underline{Aut}}

\DeclareMathOperator{\Hom}{Hom}
\DeclareMathOperator{\Gal}{Gal}
\DeclareMathOperator{\GL}{GL}
\DeclareMathOperator{\SL}{SL}

\DeclareMathOperator{\supp}{supp}

\DeclareMathOperator{\Spec}{Spec}

\DeclareMathOperator{\un}{un}

\def\no{n\textsuperscript{0}\,}

\newtheorem{mydef}[equation]{Definition}
\newtheorem{lem}[equation]{Lemma}
\newtheorem{thm}[equation]{Theorem}

\newtheorem{conj}[equation]{Conjecture}

\newtheorem{prop}[equation]{Proposition}

\newtheorem{rem}[equation]{Remark}

\newtheorem{cor}[equation]{Corollary}

\newtheorem{exam}[equation]{Example}

\newtheorem*{theorem*}{Theorem}
\newtheorem{quest}[equation]{Question}
\numberwithin{equation}{subsection}
\usepackage{blindtext}

  \DeclareFontFamily{U}{wncy}{}
    \DeclareFontShape{U}{wncy}{m}{n}{<->wncyr10}{}
    \DeclareSymbolFont{mcy}{U}{wncy}{m}{n}
    \DeclareMathSymbol{\Sh}{\mathord}{mcy}{"58} 
\title[Quantitative inverse Galois problem]{Quantitative inverse Galois problem for semicommutative finite group schemes.}
\author[Darda]{Ratko Darda}
\address{Department of Mathematics and Computer Science\\ University of Basel}
\email{ratko.darda@gmail.com}
\author[Yasuda]{Takehiko Yasuda}
\address{Department of Mathematics\\ Graduate School of Sciences\\Osaka University}
\email{yasuda.takehiko.sci@osaka-u.ac.jp}
\begin{document}
\begin{abstract}
 A semicommutative finite group scheme is a finite group scheme which can be obtained from commutative finite group schemes by iterated performing semidirect products with commutative kernels and taking quotients by normal subgroups. In this article, for an \'etale tame semicommutative finite group scheme~$G$, we give a lower bound on the number of connected $G$-torsors of bounded height (such as discriminant). % prove that a quantitative version of  the inverse Galois problem admits a positive solution for a semicommutative tame \' etale finite group scheme~$G$ over a global field, i.e. such~$G$ admits a connected $G$-torsor. Moreover, the proof gives us a lower bound on the number of connected $G$-torsors of bounded height. %Let~$G$ be a non trivial finite \' etale tame $F$-group scheme. We define height functions on the set of $G$-torsors over~$F,$ which generalize the usual heights such as discriminant. As an analogue of the Malle conjecture for group schemes, we formulate a conjecture on the asymptotic behaviour of the number of $G$-torsors over~$F$ of bounded height. The conjectured asymptotic is proven for the case when~$G$ is commutative. %Furthermore, an equidistribution of $G$-torsors is established. % in the product space $\prod_{\vMF}BG(\Fv)$. 
 % (the finite group schemes which can be obtained from finite commutative group schemes by performing iteratively semidirect products and taking quotients). %in the weak Malle conjecture for the alternating group~$\Akk$ embedded in its degree~$4$ and degree~$12$ permutation representations.
\end{abstract}
\maketitle
\section{Introduction}
\subsection{Inverse Galois problem for finite group schemes}\label{secfir}
One of the most famous questions in number theory is the {\it inverse Galois problem}, which asks whether every finite group~$G$ is realizable as the Galois group of a finite extension of the field~$\QQ$. This is a largely open question. It is classically known to admit an affirmative answer for~$G$ commutative, for $G=\mathfrak S_n$ symmetric, for $G=\mathfrak A_n$ alternating, etc. Without any modifications the question can be asked for other fields and in this paper we deal with the case of a global field~$F$. In this context, one has (a generalization of) the celebrated Shafarevich theorem \cite[Theorem 9.6.1]{CohomologyNF} which states that every solvable~$G$ is a Galois group of a finite extension of~$F$. 

If $K/F$ is an extension, then it is a Galois extension with the Galois group~$G$ if and only if $\Spec(K)\to \Spec (F)$ is a {\it connected} $G$-torsor. For non-constant finite group schemes~$G$, one can thus ask:
\begin{quest}\label{qjedan} Let~$F$ be a global field. Does every finite $F$-group scheme~$G$ admits a connected $G$-torsor? 
\end{quest}
Although the question is a very natural one, to our knowledge, in this form it was only asked in Section  ``The inverse problem of Galois theory for torsors" of \cite{MR3971189} by Cassou-Nogu\`es, Chinburg, Morin and Taylor, where an affirmative answer is provided for the case~$F$ is a number field and $G=\muu_m$ is the finite group scheme of $m$-th roots of unity. In this article, we will always assume that, besides being finite, the group schemes are {\it \' etale} and {\it tame} (i.e. if the characteristic of~$F$ is positive, then the cardinality of~$G$ is coprime to the characteristic). In the literature one can find some other properties of a finite \' etale tame $F$-group scheme~$G$ which imply a positive solution to Question \ref{qjedan} such as the following one. (Some of the next implications, even though well known to experts, were not written down explicitly in the literature unless~$G$ is assumed to be constant, and we dedicate Appendix for the proofs without the assumption.)
\begin{mydef}\label{defweakweak}
Denote by $BG(F)$ the set of $G$-torsors over~$F$ and for a place~$v$ of~$F$, denote by $BG(F_v)$ the set of $G$-torsors over the completion~$F_v$. We say that the weak weak approximation is valid for~$G$, if there exists a finite set~$S_0$ of places of~$F$, such that for every finite subset~$S$ of places of~$F$ which is disjoint from~$S_0$, we have that the canonical map 
\begin{equation}\label{map}
BG(F)\to\prod_{v\in S}BG(F_v)
\end{equation}
is surjective.
\end{mydef}
Harari defines a weaker property, the {\it hyperweak approximation}, which also implies the existence of a connected $G$-torsor. 
The property is stable for semidirect products with commutative kernel and for taking quotients by normal subgroups \cite[Proposition 2 and 3]{haraprox}, hence, the hyperweak approximation is satisfied for~$G$ \' etale, tame and {\it semicommutative} (see Section \ref{conntent}). Thus Question \ref{qjedan} admits a positive answer for such~$G$. If~$F$ is a number field, it follows from a theorem of Harpaz and Wittenberg prove \cite[Theorem B]{harpwit} that~$G$ \' etale and {\it hypersolvable} (admitting composition series by finite subgroup schemes with cyclic consecutive quotients) satisfies the weak weak approximation and hence Question \ref{qjedan} admits an affirmative answer also for such~$G$. %and constant this is a slight modification of the result of Harpaz and Wittenberg.d
%
%Let us mention that an affirmative answer to Question \ref{qjedan} for every \' etale and tame finite $F$-group scheme would follow from an affirmative answer to Colliot-Th\'el\`ene's conjecture that {\it Brauer-Manin obstruction is the only one to the weak approximation} for rationally connected varieties \cite[Conjecture ]{brauergrothendieck}.
\subsection{Quantitative aspect} A ``quantitive" version of the inverse Galois problem is given by the {\it Malle conjecture}. We have a {\it height} function~$H:BG(F)\to\RR_{>0}$ and we count how many $X\in BG(F)$ satisfy $H(X)<B$, where $B>0$.
\begin{conj}[{Malle \cite{Malle}}] \label{classmalle}
Let~$G$ be a non-trivial tame constant group which is embedded as a transitive subgroup of the group of permutations~$\mathfrak S_n$ for some $n\geq 1$. Let $G_0\subset G$ be a stabilizer of a point in $\{1\doots n\}$ for the action of~$G$ induced by the embedding. For a $G$-torsor~$X$, we write $H(X):= \Delta(X/{G_0})$, where~$\Delta$ denotes the norm of the discriminant. %For an extension $K/F$ we write $N(\Delta(K/F))$ for the norm of its discriminant and $\Gal(K/F)=G$ if the Galois group of the Galois closure of~$K/F$ is isomorphic to~$G$ as a permutation group.
One has that\begin{equation*}\#\{X\in BG(F)|X \text{ is connected and } H(K)<B\}\asymp_{B\to\infty}  B^{a(H)}\log(B)^{b(H)-1},\end{equation*}
for some explicit invariants~$a=a(H)$ and~$b=b(H)$.
\end{conj}
(This is a slightly different statement from the usual statement of the conjecture, and the equivalence with the usual one is explained in \cite[Paragraph 1.1.2]{commcase}).  Over~$\QQ$, the conjecture is known to be true for~$G$ commutative embedded in its regular representation,~$G=S_3,S_4, S_5$ embedded in its standard representations, etc (see \cite{Wright}, \cite{dave}, \cite{densityquartic}, \cite{densityquintic}). It  is significantly harder than the inverse Galois problem: e.g. it is unknown for some usual groups such as $G=\mathfrak A_4$ embedded in its standard or regular representation. Conjecture~\ref{classmalle} admits counterexamples as shown by Kl\"uners \cite{Kluners}. 
For some~$G$, %(e.g. if $F=\QQ$ and~$G=D_{2p}$ is the dihedral group of order $2p$, with~$p$ prime), 
 only upper and lower bounds on the number of $G$-torsors ($G$-extensions) of bounded height are known. For the case of a number field, based on Shafarevich theorem, Alberts establishes in \cite{statfirstg} a lower bound of the form $\gg B^a$, with $a>0$ for every solvable group~$G.$ 
 
 In our previous article \cite{commcase}, we proposed a version of Malle conjecture for non-constant finite group schemes~$G$. % of bounded {\it height} (that is of bounded ``size" of a torsor, an example of a height is the (norm of) the discriminant). 
 %We have proposed a prediction \cite[Conjecture 2.6.5]{commcase} for the number of {\it secure} $G$-torsors which generalizes the {\it Malle conjecture} (which deals only with the constant case).
%Let us denote by~$BG(F)$ the set of $G$-torsors.
\begin{conj}\label{starconj}
Let~$G$ be a non-trivial finite \' etale tame $F$-group scheme. Let~$H:BG(F)\to\RR_{>0}$ be a height. We define invariants~$a(H)$ and~$b(H)$ as in Definition \ref{onlydef}. One has that $$\#\{x\in BG(F)|\text{$x$ is secure and }H(x)\leq B\}\asymp_{B\to \infty} B^{a(H)}\log (B)^{b(H)-1}.$$
\end{conj}
(The definition of ``secure" can be found in \cite[Definition 2.6.3]{commcase} and the notion serves to avoid counterexamples. When~$G$ is commutative, every $G$-torsor is secure.) The conjecture is a special case of a {\it stacky Batyrev-Manin conjecture} \cite[Conjecture 9.15]{dardayasudabm}. Note that we have {\it not} imposed a connectivity condition. Conjecture~\ref{starconj} has been verified in \cite[Theorem 1.3.2]{commcase} for~$G$ commutative. %verified for~$G$ commutative in \cite[Theorem 1.3.2]{commcase}. 
However, it may happen that a positive proportion of $G$-torsors is not connected (e.g. this happens when $G=\muu_m$ is the group scheme of $m$-th roots of unity, as remarked in \cite[Remark 9.2.7.4]{darda:tel-03682761}). Thus, {\it a priori}, Conjecture \ref{starconj} does not imply the existence of a single connected $G$-torsor.
\subsection{Content}\label{conntent} The principal result of this article is a %(``quantitive") 
quantitative solution to the inverse Galois problem for {\it semicommutative} finite group schemes. These are the finite group schemes which can be obtained from finite commutative group schemes by iterated performing semidirect products with commutative kernels and taking quotients by normal subgroups (for details, see Definition \ref{semcomdef}). The constant semicommutative groups are precisely those which can be realized as Galois groups by successive solution to {\it split embedding problems with abelian kernels} and taking intermediate Galois extensions \cite[Chapter IV, Section 2.2]{inversegalois}. The methods of realization, however, do not work for the non-constant case.

Let us first suppose that~$G$ is commutative. Then, an assertion \cite[Theorem 1.3.3]{commcase}, which is stronger than Conjecture \ref{starconj}, is valid: it allows to determine the asymptotic behaviour after having fixed certain local conditions. We will show that the stronger statement, together with  Lemma~\ref{lemcony} which gives local conditions which force torsors to be connected, implies the existence of (infinitely many) connected torsors. More precisely, we obtain that:
\begin{thm}\label{connintro} Suppose that~$G$ is a non-trivial commutative finite \' etale tame group scheme. One has that $$\#\{x\in BG(F)|x\text{ is connected}, H(x)\leq B\}\asymp_{B\to\infty} B^{a(H)}\log(B)^{b(H)-1}.$$
%In particular, there exist infinitely many connected $G$-torsors.
\end{thm}
We mention that in \cite[Theorem 9.2.7.3]{darda:tel-03682761}, the first named author develops the precise asymptotic behaviour (with the leading constant) for the case~$F$ is a number field and $G=\muu_m$ under additional assumption that $4\nmid m$ or that $\sqrt{-1}\in F$.
%When~$F$ is a number field, this generalizes the result of \cite[Proposition 9.2.7.1]{darda:tel-03682761} for the case $G=\muu_m$ is the group scheme of $m$-th roots of unity. In this situation, with the additional assumption that $i=\sqrt{-1}\in F$ or that $4\nmid m$, it is possible to derive a precise asymptotic formula for the number of connected $\muu_m$-torsors \cite[Theorem 9.2.7.3]{darda:tel-03682761}.

Let us now treat the semicommutative case. A semicommutative finite \' etale group scheme~$G$ can be written as $G=\langle A, K\rangle$, where~$\iota:A\hookrightarrow G$ is normal and commutative and~$K\lneq G$ is semicommutative. We establish a similar bound to Alberts' bound for solvable constant groups:
\begin{thm}\label{semicintro}
Suppose that~$G$ is a non-trivial semicommutative finite \' etale and tame $F$-group scheme. Write~$G=\langle A, K\rangle$ as above. There exists $C>0$ such that $$\#\{x\in BG(F)| x\text{ is connected, } H(x)\leq B\}\geq  CB^{a(\iota^*H)},$$where $\iota^*H$ is the pullback height (defined precisely in Paragraph \ref{subheight}).% In particular, there exist infinitely many connected $G$-torsors. 
\end{thm}
The obtained lower bounds may be as good as in the {\it weak Malle conjecture} (which predicts that the number grows at least as $CB^{a(H)}$ for some $C>0$), as the following example shows. The constant alternating group $G=\mathfrak A_4$ is semicommutative (non-constant examples with $G(\oF)=\mathfrak A_4$ do exist, as discussed in Example~\ref{altasexam}). Our result implies that if the characteristic of~$F$ is not~$2$ or~$3$, the number of $\Akk$-fields of bounded discriminant is growing at least as $C B^{\frac{1}{2}}.$  For the case~$F$ is a number field this was established in \cite[Corollary~1.8]{statfirstg}. %The result gives the lower bound predicted by a weaker version of the Malle conjecture \cite{mallefirst} which predicts that the number of extensions of bounded discriminant is at least $CB^{a(G\hookrightarrow S)}$. 
\subsection{Acknowledgements} This work was supported by JSPS KAKENHI Grant Number JP18H01112. This work has been done during a post-doctoral stay of the first named author at Osaka University. During the stay, he  was supported by JSPS Postdoctoral Fellowship for Research in Japan. The authors would like to thank to Matthieu Florence and Giancarlo Lucchini Arteche for useful comments and suggestions.
\subsection{Notations} We will use notation~$F$ for a global field. We denote by~$M_F$ (respectively, by~$M_F^0$ and by~$M_F^\infty$) 
the set of its places (respectively, of its finite and infinite places). %We denote by~$\OO_F$ the ring of integers of~$F$.

We fix algebraic closures of~$F$ and of~$F_v$ for $\vMF$ and embeddings of the algebraic closure of~$F$ in each of the algebraic closures of~$F_v$. We denote by~$\overline F$ and for $\vMF$ by $\overline{F_v}$ the separable closure of~$F$ in and~$F_v$ in the chosen algebraic closures. The notation~$\Gamma_F$ and~$\Gamma_v$ will be used to denote the absolute Galois group of~$F$ and~$F_v$, respectively. For a finite place~$v$, we denote by~$\Gamma_v^{\un}$ the Galois group $\Gal(\Fv^{\un}/\Fv)$, where~$\Fv^{\un}$ is the maximal unramified extension of~$\Fv$ and by~$q_v$ the cardinality of the residue field at~$v$.
%We write $\AAA$ for the ring of the adeles of $F$.
%If~$v$ is a finite place of~$F$, we denote by~$\Ov$ the ring of integer of~$F_v$. %, by~$\kappa_v$ the residue field at~$v$ and by~$q_v$ its cardinality.  In the function field case, we write~$q$ for the cardinality of the field of constants of~$F$. 

Let $f, g:\RR_{\geq 0}\to\RR_{\geq 0}$ be two functions, such that for $B\gg 0$ one has that $g(B)\neq 0$. %We write $f\sim_{B\to\infty}g$ if $\lim_{B\to\infty}\frac{f(B)}{g(B)}=1$. 
We write $f\asymp_{B\to\infty}g$ if there are constants $C_1, C_2>0$ such that for every~$B$ big enough one has that $$ C_1 g(B)\leq f(B)\leq C_2 g(B).$$
\section{Notions}\label{semicomheights}
We recall some notions and results from \cite[Section 2]{commcase}. Let~$G$ be a non-trivial finite tame $F$-group scheme.
\subsection{Heights}\label{subheight}
 Let~$e$ be the exponent of~$G(\oF)$ and let $\muu_e$ be the group scheme of $e$-th roots of unity. %One has that $G(\overline F)$ is a finite $\Gamma_F$-group. 
The group~$G(\oF)$ acts on the $\Gamma_F$-group $\Hom (\muu_e, G(\oF))$ by conjugation $$h\cdot (x\mapsto g):=x\mapsto (hgh^{-1}).$$ The action preserves $\Gamma_F$-orbits and the identity element. We let~$G_*$ be the finite pointed $F$-scheme given by the $\Gamma_F$-pointed set $$\Hom(\muu_e, G(\oF))/G(\oF).$$  %Note that the cardinality of~$G_*$ is at least~$2$.
For a closed immersion $G\hookrightarrow R$, we have a pointed morphism $G_*\to R_*$ of trivial kernel (but not necessarily injective). The following definitions are from \cite[Paragraph 2.3.1]{commcase}.
\begin{mydef}\label{onlydef}
\begin{enumerate}
\item We call a $\Gamma_F$-invariant function $c:G_*(\oF)\to\RR_{\geq 0},$ which satisfies that $c(x)=0$ if and only if $x=1_{G_*(\oF)}$ is the distinguished element in~$G_*(\oF)$, a counting function.
\item Let $c:G_*(\oF)\to\RR_{\geq 0}$ be a counting function. We define \begin{align*}
a(c)&:=\big(\min_{x\in G_*(\oF)-1_{G_*(\oF)}}c(x)\big)^{-1}\in\RR_{>0},\\b(c)&:=\#\big\{x\in G_*(\oF)\big| c(x)=a(c)^{-1}\big\}.
\end{align*}
\item If $\iota: G'\hookrightarrow G$ is a closed immersion of a non-trivial subgroup scheme, and $c: G_*(\oF)\to \RR_{\geq 0}$ a counting function, then we set $\iota^*c:=c\circ ((G')_*\to G_*)$  (it is a counting function).
\end{enumerate}
\end{mydef}
We denote by $BG(F)$ (respectively, for $\vMF$ by $BG(F_v)$) the pointed set of $G$-torsors over~$F$ (respectively, over~$F_v$). The $\Gamma_F$-group $G(\oF)$ becomes, using inclusions $\Gamma_v\hookrightarrow\Gamma_F,$ a $\Gamma_v$-group for $v\in M_F.$ For $K\in \{\Gamma_F\}\cup\{\Gamma_v\}_{\vMF}$, we denote by $Z^1(K, G(\oF))$ the set of continuous crossed homomorphisms $f: K\to G(\oF).$
There exist canonical pointed bijections \begin{align*}BG(F)&=Z^1(\Gamma_F, G(\oF))/\sim\hspace{0.2cm}=:H^1(\Gamma_F, G(\oF)),\\
BG(F_v)&=Z^1(\Gamma_v, G(\oF))/\sim\hspace{0.2cm}=:H^1(\Gamma_v, G(\oF)), \hspace{1cm}(v\in M_F)
\end{align*}
where~$\sim$ is defined via $$f\sim f'\iff \exists g\in G(\oF):\forall\gamma\in\Gamma_F: f'(\gamma)=g^{-1} f(\gamma) (\gamma\cdot g),$$ and analogously for $\vMF$. Let~$\Sigma_G$ be the finite set given by the places~$v$ such that~$G(\oF)$ is ramified or not tame at~$v$ (that is, $\gcd (q_v, \# G(\oF))>1$). Whenever $\vMF-\Sigma_G-M_F^{\infty}$, we have a canonical map of pointed sets $$\Psi^G_v:BG(F_v)\to G_*(\oF),$$the kernel of which is $$BG(\Ov):=H^1(\Gamma_v^{\un}, G(\oF))\subset H^1(\Gamma_v,G(\oF))=BG(F_v).$$If $x\in BG(F)$, then for almost all finite~$v$, the image of~$x$ for the map $BG(F)\to BG(\Fv)$ lies in $BG(\Ov)$, hence, for almost all finite~$v$, one has that~$x$ is in the kernel of the composite map $$BG(F)=H^1(\Gamma_F, G(\oF))\to H^1(\Gamma_v, G(\oF))= BG(F_v)\xrightarrow{\Psi_v^G} G_*(\oF).$$
\begin{mydef}\label{definitionofheight}Let $c:G_*(\overline F)\to\RR_{\geq 0}$ be a counting function. Let $M_F^{\infty}\cup \Sigma_G\subset\Sigma\subset M_F$ be a finite set of places.  For $v\in\Sigma$, we let $c_v: BG(F_v)\to\RR_{\geq 0}$ be functions and for $v\in M_F-\Sigma$ let us set $$c_v=c\circ \Psi^G_v:BG(F_v)\to \RR_{\geq 0}.$$ For $v\in M_F,$ we denote by~$H_v$ the function $$H_v:BG(F_v)\to\RR_{>0}\hspace{1cm} x\mapsto q_v^{c_v(x)}.$$ The function $$H=H((c_v)_v):BG(F)\to ~\RR_{>0}\hspace{1cm}x\mapsto\prod_{\vMF}H_v(x_v),$$where~$x_v$ is the image of~$x$ for the map $BG(F)\to BG(F_v)$, is called the height function defined by~$(c_v)_v$ (sometimes simply the height). We say that~$c$ is the type of~$H$. We set \begin{align*}
a(H)&:=a(c)\\
b(H)&:=b(c).
\end{align*}
\end{mydef}
The quotient of two heights is a function which is bounded from above and below by positive constants.  If $\iota:G'\hookrightarrow G$ is a closed immersion of a non-trivial subgroup, then we define~$\iota^*H$ to be the function $BG'(F)\to BG(F)\xrightarrow{H}\RR_{>0},$ which turns out to be a height on~$BG'(F)$.
\subsection{Twists} The references for this paragraph are \cite[Paragraph 2.2.2, Lemma 2.2.6, Lemma 2.5.4]{commcase}.
Let $\sigma\in Z^1(\Gamma_F, G(\oF))$ be a cocycle. We define~$\sG $ to be the finite group scheme  which corresponds to the $\Gamma_F$-action on~$G(\oF)$ obtained by twisting by~$\sigma$: $$\gamma\cdot g:= \sigma(\gamma)g\sigma(\gamma)^{-1},\hspace{1cm}\gamma\in\Gamma_F, g\in G(\oF).$$
  There exists a canonical bijection $\lambda_{\sigma}:B(\sG)(F)\to BG (F)$, induced by \begin{align*}\Lambda_{\sigma}:Z^1(\Gamma_F, \sG(\oF))&\to Z^1(\Gamma_F, G(\oF)).\\
 f&\mapsto f\cdot\sigma.
 \end{align*}
 One has a canonical identification $(\sG)_*=G_*.$ If $H:BG(F)\to \RR_{>0}$ is a height, then $H\circ\lambda_{\sigma}:B(\sG)(F)\to\RR_{>0}$ is a height. Moreover, one has that 
\begin{align*}
a(H\circ\lambda_{\sigma})&=a(H)\\
b(H\circ\lambda_{\sigma})&=b(H).
\end{align*}
 If~$R$ is another non-trivial finite \'etale tame $F$-group scheme and $\phi:G\hookrightarrow R$ a homomorphism which is a closed immersion, we may write~$\prescript{}{\sigma}R$ for$~\prescript{}{\phi(\oF)\circ\sigma}R$. We have a closed immersion $\sG\to\prescript{}{\sigma}R$, and the induced morphism $(\sG)_*\to (\prescript{}{\sigma}R)_*$ coincides with the morphism $G_*\to R_*.$
%\subsection{Principal results from \cite{v3}}
%We recall a weaker version of our principal result from \cite{v3}. Let~$G$ be a commutative non-trivial finite \' etale tame $F$-group scheme. By \cite[Theorem, ]{CohomologyofNF} there exists a finite set of places~$\Sigma_G\subset M_F$ such that for every finite set~$S\subset M_F$ disjoint from~$\Sigma$,  the canonical map $$BG(F)\to\prod_{v\in S} BG(F_v)$$is surjective. 
%\begin{thm}
%Let $v_1\doots v_k\in M_F-\Sigma_G$ and for $i=1\doots k$, let $x_i\in BG(F_{v_i})$. There exists $C>0$ such that for every $B>0$. %\end{thm} 
%One has a canonical identification of pointed sets $(\sG)_*=G_*$. 
\section{Semicommutative groups}\label{sectionstatements}
In this section we prove our principal results.
\subsection{Commutative case}\label{conntorsor}  We prove our main result for the commutative case. 
\begin{lem}\label{critconn}
Let~$J$ be a finite \' etale $F$-group scheme. Let~$X\in BJ(F)$ and let~$x\in Z^1(\Gamma_F, J(\oF))$ be its lift. Suppose that there exists a finite set of finite places $\{v_1\doots v_k\}$ of~$F$ such that for  $1\leq i\leq k$ one has that
\begin{enumerate}
\item the finite group scheme~$J_{F_{v_i}}$ is constant;
\item one has that $J(\oF)=\langle x(\Gamma_{v_i}) \rangle_{i=1}^k.$
\end{enumerate}
Then~$X$ is connected.
\end{lem}
\begin{proof}
We fix a bijection  $X(\oF)\xrightarrow{\sim} J(\oF)$ and identify the set~$X(\oF)$ with~$ J(\oF)$ via this bijection. The action on~$X(\oF)$ is given by$$\gamma\cdot g=(x(\gamma))(\gamma(g))\hspace{1cm}\gamma\in\Gamma_F, g\in X(\oF)=J(\oF).$$One has that~$X$ is connected if and only if~$X(\oF)$ is a transitive $\Gamma_F$-set, so let us prove the latter. Let~$g_1,g_2\in X(\oF)$ and let $g=(g_2)(g_1)^{-1}$. By the second assumption, there exists a finite product of~$\prod x(\gamma_j),$ where $\gamma_j\in \Gamma_{v_1}\cup\cdots\cup\Gamma_{v_k}$ such that $\prod x(\gamma_j)=g.$ Note that for $\gamma_1, \gamma_2\in\Gamma_{v_1}\cup\cdots\cup\Gamma_{v_k}$ one has that $$x(\gamma_1\gamma_2)=(x(\gamma_1))(\gamma_1\cdot x(\gamma_2))=(x(\gamma_1))(x(\gamma_2)).$$ Hence, $x\big(\prod \gamma_j\big)=g$. We deduce that $$\big(\prod \gamma_j\big)\cdot g_1=\big(x\big(\prod \gamma_j\big)\big)\big(\big(\prod \gamma_j\big)\big(g_2\big)\big)=gg_2=g_1.$$ The action is thus transitive and the statement follows.
\end{proof}
\begin{lem}\label{lemcony} Let $G$ be a finite \' etale $F$-group scheme and let $i:BG(F)\to\prod_{\vMF}BG(F_v)$ be the diagonal map. Let~$\Sigma$ be a finite set of places of~$F$. There exists a finite subset $T\subset M_F^0-\Sigma,$ elements $y_v\in BG(\Ov)$ for $v\in T,$ such that every $x\in BG(F)$, with $$i(x)\in \bigg(\prod_{v\in T}\{y_v\}\times\prod_{v\in M_F-T}BG(F_v)\bigg),$$ is connected.
\end{lem}
\begin{proof}
\begin{enumerate}
\item First, we prove that for every $1\neq g\in G(\oF)$ we can associate a finite place~$v_g$ of~$F$ such that the following conditions are verified:
\begin{itemize}
%\item the group $\langle g\rangle\subset G(\oF)$ is a Galois group over~$F_{v_g}$;
%\item the map $BG(F)\to \prod_{1\neq g\in G(\oF)}BG(F_{v_g})$ is surjective;
\item for every $g\in G(\oF)-\{1\}$ one has that $v_g\not\in\Sigma$;
\item one has that~$G_{F_{v_g}}$ are constant finite group schemes;
\item one has that $v_g\neq v_{g'}$, whenever $g\neq g'$.
\end{itemize}
%By \cite[Theorem 9.2.3 (vii)]{CohomologyNF}, the first condition will be  valid if each~$v_g$ is such that~$G$ is unramified and tame at~$v_g$. 
Indeed, there exists a finite Galois extension~$K/F$ contained in~$\overline F$ such that~$\Gamma_F$ acts on~$G(\oF)$ via the Galois group~$\Gal (K/F)$. There exist infinitely many places~$v$ such that~$K\subset F_v$. (We write~$K=F(a)$ and let~$p_a$ be the minimal polynomial of~$a$ over~$F$. There are infinitely many~$v$ such that $v(p_a(t))>0$ for some~$t\in \OO_F$, where~$\OO_F$ is the ring of integers of~$F$. For any such~$v$ which satisfies that for every coefficient~$b_i$ of~$p_a$ one has~$v(b_i)=0$, by Hensel's lemma \cite[Chapter II, Lemma 4.6]{Neukirch}, the polynomial~$p_a$ admits a root in~$F_v$.) For such places~$v$ one has that~$G_{F_v}$ is constant. The claim follows.
\item Now, for every~$1\neq g\in~G(\oF)$, we fix a homomorphism $$\Gamma_{v_g}\to\Gamma_{v_g}^{\un}=\widehat{\ZZ}\to\langle g\rangle\subset G(\oF).$$This defines a $G_{F_{v_g}}$-torsor~$y_g$ such that $y_g\in BG(\Ov)$. Consider the open $$U:=\prod_{1\neq g\in G(\oF)}\{y_g\} \times\prod_{v\in (M_F-\{v_g|1\neq g\in G(\oF)\})}BG(F_v)\subset\prod_{\vMF}BG(F_v),$$which is also closed. By applying Lemma \ref{critconn} to $J=G$ and to the set of places $T:=\{v_g|g\in G(\oF)-1\},$ we have that if $i(x)\in U,$ then~$x$ is connected.
\end{enumerate}
\end{proof}
\begin{thm}\label{numboffieldcom} Let~$G$ be a commutative non-trivial finite \'etale and tame $F$-group scheme. Let~$H$ be a height having on $BG(F)$. One has that $$\#\{x\in BG(F)|x\text{ is connected}\}\asymp_{B\to\infty} B^{a(H)}\log(B)^{b(H)-1}.$$
\end{thm}
\begin{proof}
%The proof is a generalization of the proof of \cite[Proposition 9.2.7.1]{darda:tel-03682761}. 
Note that it suffices to assume that~$H$ is a normalized height, i.e. that $a(H)=1$. Indeed, for every height~$H$, one has that $H^{\frac{1}{a(H)}}$ is a normalized height and thus the claim for a normalized height then immediately implies the claim for a non-normalized height.

It follows from \cite[Theorem 9.2.3 (vii)]{CohomologyNF} that there exists a finite set of places~$\Sigma$ such that whenever $\Sigma'\subset M_F-\Sigma$ is finite, one has that the canonical map $BG(F)\to \prod_{v\in\Sigma'} BG(F_v)$ is surjective. Let~$T$ be and $(y_v)_{v\in T}\in\prod_{v\in T} BG(F_v)$ be given by applying Lemma~\ref{lemcony} to the finite group scheme~$G$ and the set of places~$\Sigma$. We set $$U:=\prod_{v\in T}\{y_v\}\times \prod_{v\in M_F-T}BG(F_v).$$
By construction one has that $i(BG(F))\cap U\neq\emptyset.$ In \cite[Lemma 3.5.1]{commcase}, we have defined a Radon measure~$\omega_H$ on the product space $\prod_{\vMF}BG(F_v)$. We have proven in \cite[Lemma 3.5.2]{commcase} that $$\supp(\omega_H)=\overline{i(BG(F))}.$$ As~$U$ is an open neighbourhood of a point in~$\supp(\omega_H)$, we deduce that $\omega_H(U)>0$. 
 Now, the statement follows by applying \cite[Theorem 3.5.5]{commcase} and \cite[Theorem 3.5.6]{commcase} to the characteristic function~$\mathbf 1_U$ of the open set with empty boundary having positive $\omega_H$-volume~$U$.
\end{proof}
\subsection{Semidirect products}
In this subsection, we will study torsors of semidirect products of finite \' etale tame $F$-group schemes.

Let~$A$ and~$K$ be finite \' etale tame $F$-group schemes. Suppose we are given an $F$-homomorphism $\phi:K\to\Autt(A)$, where $\Autt(A)$ is the finite \' etale $F$-group scheme given by the $\Gamma_F$-group $\Aut(A (\oF))$ and the following action $$\gamma\cdot t=\gamma\circ t\circ\gamma^{-1}\hspace{1cm}(\gamma\in\Gamma_F, t\in \Aut(A(\oF)).$$
We let $A\rtimes _{\phi}K$ be the group scheme given by the group $N(\oF)\rtimes_{\phi(\oF)} K(\oF)$ which is endowed with the following $\Gamma_F$-action $\gamma\cdot (n_0,h_0)=(\gamma(n_0),\gamma(h_0))$. 

Let $\theta\in Z^1(\Gamma_F, K(\oF))$ be a crossed homomorphism and let~$\Theta$ be the $K$-torsor defined by~$\theta$. The image of~$\theta$ for the map $Z^1(\Gamma_F, K(\oF))\to Z^1(\Gamma_F, K(\oF)\rtimes_{\phi(\oF)}H(\oF))$ induced by $$K(\oF)\to A(\oF)\rtimes _{\phi(\oF)}K(\oF)\hspace{1cm}h\mapsto (1,h)$$ is the map $\sigma:=\gamma\mapsto (1,\theta(\gamma))$.  Let~$\sA$ be the group subscheme of $\prescript{}{\sigma}(A\rtimes_{\phi}K)$ corresponding to the subgroup $A(\oF)$ which is $\Gamma_F$-invariant for the twisted action. %for the normal closed subgroup of $\prescript{}{\sigma}(N\rtimes_{\phi} K)$ given by $\Gamma_F$-invariant normal subgroup $\{(n,1)|n\in N(\oF)\}$ of $(\prescript{}{\sigma}(N\rtimes_{\phi} K))(\oF)$.
\begin{lem}\label{canmaptw}
The canonical map $$u^{\phi}_\sigma:B(\sA)(F)\to B(\prescript{}{\sigma}(A\rtimes _{\phi}K))(F)\xrightarrow{\lambda_{\sigma}}B(A\rtimes _{\phi}K)(F)$$ is given by $u^{\phi}_\sigma(X)= X\times_F\Theta$. 
\end{lem}
\begin{proof}
Let $X\in B(\sA)(F)$ and let $x\in Z^1(\Gamma_F, \sA(\oF))$ be a lift of~$X$. The image of~$x$ under the canonical map $$Z^1(\Gamma_F, \sA(\oF))\to Z^1(\Gamma_F, \prescript{}{\sigma}(A\rtimes _{\phi}K))$$ is given by $\gamma\mapsto(x(\gamma),1)$.  The map $\lambda_{\sigma}$ is induced by the map $$\Lambda_{\sigma}:Z^1(\Gamma_F, (\prescript{}{\sigma}(A\rtimes _{\phi}K))(\oF))\to Z^1(\Gamma_F, (A\rtimes _{\phi}K)(\oF))$$ which is given by $y\mapsto \big(\gamma\mapsto y(\gamma)\cdot (1,\theta(\gamma))\big)$. It follows that the image of~$X$ for the map~$u^{\phi}_\sigma$ is the $A\rtimes_{\phi}K$-torsor induced by the crossed homomorphism $\gamma\mapsto (x(\gamma), \theta(\gamma)).$ By \cite[Page 47]{Cohomologiegalois}, the $A\rtimes_{\phi} K$-torsor induced by $\gamma\mapsto (x(\gamma), \theta(\gamma))$ is isomorphic to $A\rtimes _{\phi}K$-torsor given by the group $A(\oF)\rtimes_{\phi(\oF)}K(\oF)$ and the following $\Gamma_F$-action: 
\begin{align*}
\gamma\cdot (n_0,h_0)&=(x(\gamma),\theta(\gamma))\cdot (\gamma(n_0),\gamma(h_0))\\
&=\bigg((x(\gamma))\bigg(\phi(\theta(\gamma))(\gamma(n_0))\bigg),\theta(\gamma)\gamma(h_0)\bigg).
\end{align*}
The $\sA$-torsor~$X$ is isomorphic to the $\sA$-torsor given by the group~$A(\oF)$ and the following $\Gamma_F$-action $$\gamma\cdot n_0=(x(\gamma) (\phi(\theta(\gamma)).\gamma(n_0))).$$
The $K$-torsor~$\Theta$ is isomorphic to the $K$-torsor defined by the group~$K(\oF)$ and the following $\Gamma_F$-action $$\gamma\cdot h_0=((\theta(\gamma))(\gamma(h_0))).$$ By comparing the actions, we see immediately that $X\times _F\Theta=u^{\phi}_\sigma(X)$. The statement is proven.
\end{proof}
The following notion is a ``quantitative" variant of the notion of weak weak approximation.
\begin{mydef}\label{defsath}
Let~$G$ be a non-trivial finite \' etale tame $F$-group scheme and let~$\alpha>0$ and $\beta\geq 0$. Let~$H:BG(F)\to\RR_{>0}$ be a height. We say that~$G$ is $(H,\alpha,\beta)$-saturated if the following condition is satisfied:
\begin{itemize}
\item there exists a finite subset~$S\subset M_F$ such that for every finite $T\subset M_F-S$ and every $(z_v)_{v\in T}\in\prod_{v\in T} BG(F_v)$, %one has that the group scheme~$\prescript{}{\sigma}N$ 
one has that there exists $C>0$ such that\begin{equation*}\#\{y\in BG(F)|y\text{ is connected}, \forall v\in T, y\otimes_FF_v\cong z_v, H(x)\leq B\}\geq C B^{\alpha}\log(B)^{\beta}
\end{equation*}
for $B\gg 0$.
\end{itemize}
\end{mydef}
\begin{rem}
\normalfont We may drop the assumption that~$y$ is connected. Indeed, it follows  from Lemma~\ref{lemcony} that one can choose finitely many local conditions at places disjoint from~$S$ which will force every $G$-torsor satisfying them to be connected. We then add the corresponding places to~$S$. 
\end{rem}
Clearly, for two heights~$H_1$ and~$H_2$ which have the same type, one has that~$G$ is $(H_1,\alpha,\beta)$-saturated if and only if it is $(H_2,\alpha,\beta)$-saturated. It is well known \cite[Theorem 9.2.3 (vii)]{CohomologyNF} that if~$G$ is commutative, then~$G$ satisfies the weak weak approximation. Moreover,  Theorem~\ref{numboffieldcom} implies that~$G$ is $(H, \alpha(H),\beta(H)-1)$-saturated.
\begin{prop}\label{torsormach}
Let~$G$ be a non-trivial finite \' etale tame $F$-group scheme. We suppose that $G=\langle A, K\rangle,$ where $A\leq G$ and $K\lneq G$ are closed subgroups, such that~$A$ is normal in~$G$ and~$K$ admits a connected torsor~$\Theta$. Let~$\phi:K\to\Autt(A)$ be the homomorphism given by the conjugation. Let~$\sigma_K\in Z^1(\Gamma_F, K(\oF))$ be a lift of~$\Theta$ and let~$\sigma$ be the image of $\sigma_K$ for the map $Z^1(\Gamma_F, K(\oF))\to Z^1(\Gamma_F, (A\rtimes_{\phi}K)(\oF))$ induced by the map $$K\to A\rtimes_{\phi}K,\hspace{1cm} k\mapsto (1,k).$$ Let $\alpha,\beta>0$. Let~$H:BG(F)\to\RR_{>0}$ be a height. Suppose that $\prescript{}{\sigma}A$ is $(H\circ u^{\phi}_\sigma, \alpha, \beta)$-saturated. %, where $f:B(\prescript{}{\sigma}A)(F)\to B(\prescript{}{\sigma}(A\rtimes_{\phi}K))(F)\to BG(F)$ is the usual map. 
There exists $C>0$ such that $$\#\{x\in BG(F)|x\text{ is connected, } H(x)\leq B\}\geq C B^{\alpha}\log(B)^{\beta}$$for $B\gg 0$.
\end{prop}
\begin{proof}
We split the proof in the several steps.
\begin{enumerate}
\item We recall a known fact: if $G_1\subset G_2$ is normal subgroup of a finite \' etale and tame $F$-group scheme~$G_2,$ the canonical map $B(G_2)(F)\to B(G_2/G_1)(F)$ is given by $x\mapsto (x/G_1)$. Indeed, let~$\widetilde x\in Z^1(\Gamma_F, G_2(\oF))$ be a lift of~$X\in B(G_2)(F)$. Its image in $Z^1(\Gamma_F, (G_2/G_1)(\oF))$ is $w\circ\widetilde x$, where $w:G_2(\oF)\to (G_2/G_1)(\oF)$ is the quotient map. The element in $B(G_2/G_1)(F)$ associated to~$w\circ\widetilde x$ is isomorphic to $\Gamma_F$-set given by $(G_2/G_1)(\oF)$ endowed with the following $\Gamma_F$-action, where $y\in G_2(\oF)$: \begin{equation*}\gamma\cdot w(y):=(w(\widetilde x(\gamma)))\gamma(w(y)))=w(\widetilde x(\gamma) \gamma (y)).
\end{equation*}On the other side, the quotient of~$x$ by~$G_2$ is isomorphic to the $\Gamma_F$-set $$\gamma\cdot w(y)=w(\gamma\cdot y)=w(\widetilde x(\gamma)\gamma(y)),$$ and the claim follows.
\item We have a map $$A\rtimes_{\phi}K\to\langle A, K\rangle=G\hspace{1cm}(a,k)\mapsto ak$$ and we denote by~$\sigma_G$ the image of~$\sigma$ for the induced map $Z^1(\Gamma_F, (A\rtimes_{\phi}K)(\oF))\to Z^1(\Gamma_F, G(\oF))$. The composite map $K\to A\rtimes_{\phi}K\to G $ is the inclusion $K\hookrightarrow G$, hence, one has that~$\sigma_G$ is precisely the image of~$\sigma_K$ for the map induced by the inclusion. It is immediate that %By Lemma \ref{doslemtwis} one has that 
$\prescript{}{\sigma_G}A=\prescript{}{\sigma}A$ and that the homomorphism $\prescript{}{\sigma_G}A\to\prescript{}{\sigma_G}G$ induced by~$\sigma_G$ is the homomorphism $\prescript{}{\sigma}A\hookrightarrow\prescript{}{\sigma}(A\rtimes_{\phi}K)\to\sG=\prescript{}{\sigma_G}G$ induced by~$\sigma$.  
It follows that the map $$u^{\phi}_\sigma:B(\prescript{}{\sigma_G}A)(F)\to B(\prescript{}{\sigma_G}G)(F)=B(\sG)(F)\xrightarrow{\lambda_{\sigma}}BG(F),$$
which by \cite[Lemma 2.6.1]{commcase} has all fibers of cardinality at most $\#G(\oF)$, coincides with the map
$$B(\prescript{}{\sigma}A)(F)\to B(\prescript{}{\sigma}(A\rtimes_\phi K))(F)\to B(\prescript{}{\sigma}G)(F)\xrightarrow{\lambda_{\sigma}}BG(F).$$
Now, \cite[Lemma 2.2.1, Part (5)]{commcase} gives that the maps coincide with the map  $$u^{\phi}_\sigma:B(\prescript{}{\sigma}A)(F)\to B(\prescript{}{\sigma}(A\rtimes_\phi K))(F)\xrightarrow{\lambda_{\sigma}}B(A\rtimes_{\phi}K)(F)\to BG(F).$$
By combining Part (1) together with Lemma~\ref{canmaptw}, we obtain that the map~$f$ is given by $x\mapsto (x\times_F\Theta)/N,$ where~$N$ is the kernel of $A\rtimes_{\phi}K\to G.$ It follows, in particular, that the image~$u^{\phi}_\sigma(x)$ is connected if $x\times_F\Theta$ is connected.  %Moreover, by applying Lemma~\ref{finitefibers} to the first expression for the map~$f$, we see that it has fibers bounded by the cardinality $\# G(\oF)$.
\item By an abuse of notation, we may use the same letters for fields and corresponding spectra. Let~$\widetilde\Theta/F$ be the Galois closure of~$\Theta.$  Note that in order that $x\otimes_F\Theta$ is a field it suffices that~$x$ and $x\otimes_F\widetilde{\Theta}$ is a field. Let $\Theta_1\doots\Theta_k$ be the minimal subextensions of~$\widetilde{\Theta}/F$ which are strictly larger than~$F$. By \cite[Chapter~V, \S10, \no 8, Theorem 5]{algebrequatre}, if~$x$ is a field, one has that $x\otimes_F\widetilde{\Theta}$ is a field if and only if~$x$ does not contain any of the subfields $\Theta_1\doots\Theta_k$. If $\Theta_j\subset x$ then for every $v\in M_F^0$, one has that $\Theta_j\otimes_FF_v\subset x\otimes_FF_v$.
\item For every~$j=1\doots k,$ it follows from \v Cebotarev theorem \cite[Theorem 9.1.3]{CohomologyNF} that there exist infinitely many places $v\in M_F^0,$ such that~$\Theta_j$ does not have a degree~$1$ place over it. (We recall the implication. Let~$\widetilde\Theta_i$ be the Galois closure of~$\Theta_i$. By \cite[Lemma 13.5]{Neukirch}, which is stated only for number fields, but the presented proof is valid for function fields as well, the Dirichlet density that~$v$ does admit a degree~$1$ place over it is equal to $$\frac{\#\bigcup_{g\in \Gal(\widetilde\Theta_i/F)}g\Gal(\widetilde\Theta_i/\Theta_i)g^{-1}}{\#\Gal(\widetilde\Theta_i/F)}.$$ We verify that the last quotient is strictly less than~$1$. Indeed, there are at most $[\Gal(\widetilde\Theta_i/F):\Gal(\widetilde\Theta_i/\Theta_i)]$ conjugates of the subgroup $\Gal(\widetilde\Theta_i/\Theta_i)$ and each of them contains the element~$1\in\Gal(\widetilde\Theta_i/F)$. Hence,  \begin{align*}\#\bigcup_{g\in \Gal(\widetilde\Theta_i/F)}g\Gal(\widetilde\Theta_i/\Theta_i)g^{-1}\hskip-3cm&\\&\leq (\#\Gal(\widetilde\Theta_i/\Theta_i)-1)\cdot [\Gal(\widetilde\Theta_i/F):\Gal(\widetilde\Theta_i/\Theta_i)]+1\\&<\#\Gal(\widetilde\Theta_i/F) .
\end{align*}
 The claim follows.) Let~$v_j$ be such a place not contained in the finite set $S\subset M_F$ which is as in the Definition~\ref{defsath} (recall that~$\prescript{}{\sigma}A$ is $(H\circ u^{\phi}_\sigma, \alpha,\beta)$-saturated.) One has that $\Theta_j\otimes _FF_{v_j}$ is a product fields, none of which is isomorphic to~$F_{v_j}$. We set $$U:=\prod_{j=1}^k\{(\prescript{}{\sigma}A)_{F_{v_j}}\}\times\prod_{v\not\in\{v_1\doots v_k\}}BG(F_v)\subset \prod_{\vMF}BG(F_v).$$ 
For~$x\in B(\prescript{}{\sigma}A)(F)$ which is connected and such that $i(x)\in U$, we have that $x\not\supset \Theta_j$ because the trivial $(\prescript{}{\sigma}A)_{F_{v_j}}$-torsor has a component isomorphic to~$\Spec(F_{v_j})$. Hence, for such~$x$ one has that $x\otimes_F\Theta_j$ is a field. %For every~$i$ and~$v$ we fix a place~$v^i$ of~$\Theta_i$ above~$v$. Then $\Theta_i\otimes_FF_v\subset x_v$ implies that $(\Theta_{i})_{v^i}\subset x_v$, where the subscript~$v^i$ stands for the completion. If moreover~$v\in M_F^0$ is such that~$(\prescript{}{\sigma}A)_{\Fv}$ is constant, then $(\Theta_{i})_{v^i}$ is a quotient of~$x_v$ by certain subgroup~$C^i_v(x)\subset(\prescript{}{\sigma}A_{\Fv})(\oF)=A(\oF)$. Moreover, some~$C^i(x)$ appears infinitely many times among the groups~$C^i_v(x)$. For a subgroup~$D\subset A(\oF),$ let us denote by~$B(\prescript{}{\sigma}A)(F)^D_i$ the subset of connected $\prescript{}{\sigma}A$-torsors~$x$ for which $C^i(x)=D$.
Set $T=\{v_1\doots v_k\}.$ It follows from \cite[Lemma 2.6.1]{commcase} and the assumption that~$\prescript{}{\sigma}A$ is $(H\circ u^{\phi}_\sigma,\alpha,\beta)$-saturated that for $B\gg 0$ one has
\begin{align*}
\#\{y\in BG(F)|\text{$y$ is connected and }H(x)\leq B\}\hskip-5cm&\\&\geq
\# G(F)\cdot\# u^{\phi}_\sigma\big(\{x\in B(\prescript{}{\sigma}A)(F)|\\&\quad\quad\quad\quad x\text{ is connected, } i(x)\in U, H(u^{\phi}_\sigma(x))\leq B\}\big)\\&\geq B^{\alpha}\log(B)^{\beta},
\end{align*} 
for some $C>0$, where $i:B(\sA)(F)\to\prod_{\vMF} B(\sA)(F_v)$ is the diagonal map. The theorem has been proven.
\end{enumerate}
\end{proof}
\begin{rem}
\normalfont We note that connected $\sA$-torsors have~$\Theta$ for a {\it resolvent} (that is, the Galois closure of the corresponding extensions contain the extension corresponding to~$\Theta$). The question of counting extensions with a fixed resolvent has been studied in \cite{MR2745550},  \cite{MR3554453}, \cite{MR4133705}, etc.
\end{rem}
\subsection{Semi-commutative groups} We establish a lower bound on the number of connected torsors for {\it semicommutative} group schemes. A reference for the definition and basic properties for the constant case is \cite[Chapter IV, Section 2.2]{inversegalois}.
\begin{mydef}\label{semcomdef}
We say that a finite \' etale $F$-group scheme~$G$ is semicommutative if there exists a finite set of commutative subgroup schemes $\{A_i\}_{i=1}^m$ such that $$G=\langle A_i\rangle_{i=1}^m\text{ and }A_i\leq \mathcal N_G(A_j)\text{ whenever }i\leq j, $$
where $\langle K_i\rangle_{i=1}^m$ denotes the smallest closed subgroup scheme containing the subschemes~$K_i$ of~$G$ and~$\mathcal N_G(K)$ denotes the normalizer of~$K$, i.e. the largest closed subgroup scheme of~$G$ containing the closed subgroup scheme~$K$ as a normal subgroup.
\end{mydef}
The following characterization for the constant case is due to Dentzer.
\begin{prop}\label{charofsem} Let~$G$ be a non-trivial finite \' etale $F$-group scheme. The following conditions are equivalent.
\begin{enumerate}
\item $G$ is semicommutative.
\item There exists a commutative normal subgroup~$A$ of~$G$ and a semicommutative closed subgroup~$K\lneq G,$ such that~$G=\langle A, K\rangle$.
\item There exist a sequence~$(G_i)_{i=0}^k$ of finite \' etale $F$-group schemes  with $G_0=\{0\}$ and $G_k\cong G$, a sequence $(A_i)_{i=1}^{k-1}$ of commutative finite \' etale $F$-group schemes, a sequence of homomorphisms $(\phi_i:G_i\to\Autt(A_{i}))_{i=0}^{k-1}$ of finite $F$-group schemes and a sequence of normal subgroup schemes $\big(N_i\subset ( A_i\rtimes_{\phi_i}G_i)\big)_{i=0\doots k-1}$ such that for every $i=1\doots k$ one has that $$G_i=(A_{i-1}\rtimes_{\phi_{i-1}} G_{i-1})/N_{i-1}.$$
\end{enumerate}
\end{prop}
\begin{proof}
The proof is identical to the constant case \cite[Chapter IV, Theorem 2.7]{inversegalois}.
\end{proof}
\begin{thm}\label{prosemi}
Suppose that~$G$ is a semicommutative \' etale tame $F$-group scheme and write $G=\langle A, K\rangle,$ with~$\iota:A\hookrightarrow G$ commutative,~$K\lneq G$ semicommutative. % and $\sigma\in Z^1(\Gamma_F, G(\oF))$. %for certain $\phi:K\to\Autt(A)$ with $A,K$ commutative and~$N$ is a normal subgroup of $A\rtimes_{\phi}K$. 
Let $c:G_*(\oF)\to\RR_{\geq 0}$ be a counting function and let~$H:BG(F)\to\RR_{>0}$ be a height having~$c$ for its type. There exists $C>0$ such that 
\begin{equation*}\#\{x\in BG(F)|x\text{ is connected, }H(x)\leq B\}\\\geq C B^{a(\iota^*c)}\end{equation*}
for $B\gg 0$.
\end{thm}
\begin{proof}
%First we observe that as $\lambda_{\sigma}^{-1}:BG(F)\to BG_0(F)$ is a bijection such that the type of $H\circ\lambda_{\sigma}^{-1}$ is~$c$, it suffices to prove the statement under assumption~$\sigma=1$ and~$G_0=G$ is semicommutative. 
The proof is by induction on the cardinality of~$G$. By induction, we can suppose that there exists at least one connected $K$-torsor~$\Theta$. Let~$\sigma$ be the image of a lift of~$\Theta$ for the map $Z^1(\Gamma_F, K(\oF))\to Z^1(\Gamma_F, (A\rtimes_{\phi}K)(\oF))$. Consider the inclusion $\kappa:\prescript{}{\sigma}A\to \sG$. One has that 
$a(\iota^*c)=a(\kappa^*c)$, because the homomorphism $(\prescript{}{\sigma}A)(\oF)\to(\prescript{}{\sigma}G)(\oF)$ coincides with the homomorphism $A(\oF)\to G(\oF)$, hence the map $(\prescript{}{\sigma}A)_*(\oF)\to(\prescript{}{\sigma}G)_*(\oF)$ coincides with the map $A_*(\oF)\to G_*(\oF)$. It follows from Theorem~\ref{numboffieldcom}, that the finite group scheme~$\prescript{}{\sigma}A$ is $(H\circ g, a(\iota^*c),0)$-saturated, where~$g$ is the map $$B(\prescript{}{\sigma}A)(\oF)\to B(\prescript{}{\sigma}G)(F)\xrightarrow{\lambda_{\sigma}} BG(F).$$ By Proposition~\ref{torsormach}, we have for $B\gg 0$ that $$\#\{x\in BG(F)|x\text{ is connected}, H(x)\leq B\}\geq C B^{a(\iota^*c)}$$for some $C>0$.
\end{proof}
%\begin{rem}\label{gupal}
% \normalfont Suppose that~$F$ is a number field. Let~$G$ be an \' etale finite group scheme which can be obtained from commutative \' etale finite group schemes in finitely many steps by performing semidirect products with commutative kernel. By \cite[Theorem 1]{haraprox}, one has that~$G$ satisfies weak weak approximation. Now, for such~$G$ by the conclusion from Remark \ref{wwaconn}, we deduce the existence of a connected $G$-torsor. However, the weak weak approximation for~$G$ does not {\it a priori} imply that, for a closed normal subgroup~$N$ of~$G$, the weak weak approximation is valid for~$G/N$ (namely, there are constant groups~$G$ having normal subgroups~$N$ such that the maps $BG(F_v)\to B(G/N)(F_v)$ are not surjective for infinitely many~$v$). Consequently, \cite[Theorem 1]{haraprox} does not imply the weak weak approximation, nor the existence of a single connected torsor, for a general semicommutative finite group scheme.
% \end{rem}
\begin{exam}\label{altasexam}
\normalfont Suppose that the characteristic of~$F$ is not~$2$ or~$3$. The alternating group~$\Akk$ of order~$12$ has a normal commutative subgroup of order~$4$ which is preserved by every automorphism of~$\Akk$. It follows from this fact and Proposition~\ref{charofsem} that a finite \' etale group scheme~$G,$ for which $G(\oF)=\Akk,$ is semicommutative if and only if it contains a closed subgroup of order~$3$.   This happens e.g. when $G=\Akk$ is constant, but also for any (not necessarily constant)~$G$ of the form $G=\prescript{}{\sigma}(\mathfrak A)_4$ where $\sigma:\Gamma_F\to \mathfrak S_4$ is such that the induced action of~$\Gamma_F$ on~$\{1\doots 4\}$ fixes an element. The natural representation $\mathfrak A_4\subset \mathfrak S_4$ induces a counting function~$c:(\mathfrak A_4)_*\to \RR_{>0}$ given by $c(x)= 2$ if $x\neq 1$. We deduce, in particular, from Theorem~\ref{prosemi} that the number of $\mathfrak A_4$-fields of bounded discriminant is growing as $CB^{\frac 12}$ for some $C>0$.
\end{exam}
\appendix
\section{}
Let~$F$ be a global field. In this appendix, we will write down some examples of non-constant \'etale and tame finite $F$-group schemes for which Question \ref{qjedan} is known to have a positive answer, even though it is not explicitly written down in the literature.
\subsection{Hyperweak approximation}
Harari defines the following notion in \cite[Section 4]{haraprox}.
\begin{mydef}[Harari] \label{ljupaw}
Let~$G$ be \' etale and tame finite $F$-group scheme. We say that~$G$ satisfies the hyperweak approximation, if there exists a finite set of places~$S_0\subset M_F$ such that for every finite $S\subset M_F-S_0$ one has that the image of the canonical map $$BG(F)\to \prod_{v\in S}BG(\Fv)$$contains the subset $\prod_{v\in S}BG(\Ov)$.
\end{mydef}
 He establishes in \cite[Proposition 1]{haraprox} that if~$F$ is assumed to be a number field and~$G$ to be constant \' etale finite group scheme satisfying the hyperweak approximation, then the inverse Galois problem has an affirmative answer for~$G$. As we have said in Section \ref{secfir}, the ``constant'' assumption is redunant. (We thank Lucchini Arteche for indicating us this).
\begin{prop}\label{hypweak} Let~$G$ be an \' etale tame finite $F$-group scheme. Suppose that~$G$ satisfies the hyperweak approximation. Then~$G$ admits a connected $G$-torsor.
\end{prop} 
\begin{proof}
Let~$S_0\subset M_F$ be as in Definition \ref{ljupaw}. By Lemma~\ref{lemcony}, it is possible to choose for $g\in G(\oF)-\{1\}$ places $v_g\in M_F-S_0$ and elements $y_g\in BG(\Ov)$ such that $v_g\neq v_{g'}$ if $g\neq g'$ and such that whenever $x\in BG(F)$ satisfies that $$i(x)\in \prod_{g\in G(\oF)-\{1\}}\{y_g\}\times\prod_{v\in M_F-\{v_g|1\neq g\in G(\oF)\}}BG(\Fv),$$then~$x$ is connected. By the assumption that~$G$ satisfies the hyperweak approximation, the set of such~$x$ is non-empty. The claim follows.
\end{proof}
We list two results due to Harari. He proves them with the assumption that~$F$ is a number field, but the identical proofs work in the global field case with the tameness assumption.
\begin{prop}[{\cite[Harari, Proposition 2 and 3]{haraprox}}]Let~$G$ be \' etale and tame finite $F$-group scheme which satisfies the hyperweak approximation.
\begin{enumerate}
\item  Let~$N$ be a subgroup scheme of~$G$. The finite groups scheme $G/N$ satisfies the hyperweak approximation.
\item Let~$A$ be a commutative \' etale and tame finite $F$-group scheme and let $\phi:G\to \Autt(A)$ be a homomorphsm. The semidirect product $A\rtimes_{\phi}G$ satisfies the hyperweak approximation.
\end{enumerate}
\end{prop}
Thus, the results of Harari are sufficient to conclude the following fact.
\begin{cor}
Let~$G$ be a semicommutative \' etale tame finite $F$-group scheme. Then~$G$ admits a connected $G$-torsor.
\end{cor}
\subsection{Hypersolvable groups}We now suppose that~$F$ is a number field.
\begin{mydef}
Let~$G$ be an \'etale finite $F$-group scheme. We say that~$G$ is hypersolvable if there exists composition series of finite subgroup schemes $\{1\}=G_0\subset\cdots\subset G_k=G$ such that for every $i=1\doots k$ one has that $(G_i/G_{i-1})(\oF)\cong \ZZ/r_i\ZZ$ for some $r_i\in\ZZ_{>1}$.  
\end{mydef}
The following property is an immediate consequence of \cite[Theorem B]{harpwit} due to Harpaz and Wittenberg. (In the article, they derived it only with the assumption~$G$ is constant.) %) Recall that for a proper variety~$X$ we say that it is true that the Brauer-Manin obstruction is the only obstruction to the weak approximation if the closure of the image of the canonical map $$X(\oF)\to \prod_{\vMF}X(F_v)$$ coincides with the Brauer-Manin set of the variety \ref{}. 
\begin{prop} Let~$G$ be an \'etale finite $F$-group scheme which is hypersolvable. Then~$G$ satisfies the weak weak approximation.
\end{prop}
\begin{proof}
There exists a closed embedding $G\hookrightarrow\SL_n$ for some $n\geq 1$ (indeed, by \cite[Corollary 4.10]{milneagroups}, one can embed~$G$ to $\GL_m$ for certain~$m$ and~$\GL_m$ can be embedded in~$\SL_{m+1}$ via $A\mapsto A\oplus (\det(A)^{-1})$). By Hironaka theorem, there exists a proper smooth geometrically integral variety~$X$ which is $F$-birational to $\SL_n/G$. By \cite[Theorem~B]{harpwit}, one has that the Brauer-Manin obstruction is the only one for the weak approximation for~$X$ (that is, the closure of the image of the canonical map $X(F)\to\prod_{\vMF}X(F_v) $ coincides with the vanishing locus of the Brauer-Manin pairing). The variety~$X$ is unirational, thus by \cite[Page 347, (6)]{brauergrothendieck} satisfies the conditions of \cite[Lemma 13.3.13]{brauergrothendieck} and applying it, gives that the weak weak approximation is valid for~$X$ (that is, %, there exists a finite set of places~$T_0$ such that for every finite $T\subset M_F-T_0$, one has that the closure of
 the closure of the image of the map $X(F)\to\prod_{\vMF}X(F_v)$ is of the form $Z\times\prod_{v\in M_F-T}X(F_v)$ for some finite set of places~$T$ and some closed subset $Z\subset\prod_{v\in T}X(F_v)$). Hence, by \cite[Proposition 13.2.3]{brauergrothendieck}, one has that the weak weak approximation is valid for~$\SL_n/G$ and by the equivalence from \cite[Page 551]{haraprox}, we obtain that~$G$ satisfies the weak weak approximation. 
\end{proof}
Hence, it is immediate from Proposition \ref{hypweak} that:
\begin{cor}
Let~$G$ be a hypersolvable \'etale finite group scheme over the number field~$F$. Then~$G$ admits a connected $G$-torsor.
\end{cor}
%By combining the fact that commutative \'etale tame finite $F$-group schemes and the above results of Harari, we deduce the existence of a connected $G$-torsor for a semicomutative \'etale tame finite $F$-group scheme.
%With the assumption that~$F$ is a number field, Harari proves in \cite[Proposition 2]{haraprox} and \cite[Proposition 3]{haraprox} that the hyperweak approximation is preserved for taking quotients by normal subgroups and for taking a semidirect product with a commutative \' etale finite~$F$-group scheme as a kernel. If~$F$ is assumed to be a global field and the finite group schemes are also assumed to be tame, both claims are valid in the global field case with identical proofs as in \cite{haraprox}.
\bibliography{bibliography}
\bibliographystyle{acm}
\end{document}